\documentclass[11pt, requo]{amsart}

\usepackage[all]{xy}
\usepackage{bbm}
\usepackage{mathrsfs}
\usepackage{amsmath}
\usepackage{amsfonts}  
\usepackage{amscd}
\usepackage{amssymb}
\usepackage{verbatim}%
\usepackage{hyperref}
\usepackage{mathtools}
\usepackage{color}

\newtheorem{thm}{Theorem}[section]
\newtheorem{lem}[thm]{Lemma}
\newtheorem{coro}[thm]{Corollary}
\newtheorem{prop}[thm]{Proposition}

\newtheorem{defn}[thm]{Definition}

\numberwithin{equation}{thm}%

\begin{document}

\title{Duplex Hecke Algebras and Related Quantum Schur Duality}
\author{Chenliang Xue, An Zhang}

\thanks	{This work is supported by National Natural Science Foundation of China (NSFC) Gr:  12071136.}

\address{School of Mathematical Sciences, East China Normal University, Shanghai 200241, China}
\email{52215500003@stu.ecnu.edu.cn}
\address{School of Mathematical Sciences, East China Normal University, Shanghai, 200241, China.} 
\email{845684219@qq.com}

\begin{abstract}
We introduce the duplex Hecke algebra, which is an infinite dimensional algebra generated by two Hecke algebras. This concept originates from the degenerate duplex Hecke algebra in the theory of Schur-Weyl duality related to enhanced reductive algebraic groups. We will study the finite dimensional natural representation of the duplex Hecke algebra on tensor space and prove that the duplex Hecke algebra forms a duality with the quantum group of Levi type.
\end{abstract}
\maketitle
\section{Introduction}

An algebraic group $G$ is called a semi-reductive group if $G$ is a semi-direct product of a reductive closed subgroup $G_0$ and the unipotent radical $U$. When the underground field is of characteristic $p>0$, the study of semi-reductive algebraic groups and their Lie algebras is very important in many aspects of the representation theory (see \cite{OSY} or \cite{SXY}).

 Let $G={\rm GL}(V)$ and $\nu$ be the natural representation on $V$. Let $\underline{V}$ be a one-dimensional extension of $V$. Then we have a typical enhanced reductive algebraic group $\underline{G}=G\times_\nu V $, which is a closed subgroup of ${\rm GL}(\underline{V})$. The enhanced reductive group $\underline{G}$ is naturally a semi-reductive group. 
 By the classical Schur-Weyl duality, the study of polynomial representations of general linear groups produces Schur algebras. By analogy of this, the tensor representations of an enhanced group $\underline{G}$ naturally produce the so-called enhanced Schur algebra $\mathcal{E}(n,r)$, which is the algebra generated by the image of $\underline{G}$ in the $r$-{th} tensor representation of $\underline{V}^{\otimes r}$. 
 In order to develop the representation of enhanced Schur algebras and to investigate dualities of invariant groups and algebras in the enhanced tensor representations, the degenerate duplex Hecke algebra, denoted by $\mathcal{H}_r$, is introduced in \cite{SXY}.

The following results about degenerate duplex Hecke algebras are established in \cite{SXY}:
\[{\rm End}_{\mathbb{C}\Phi({\rm GL}_n\times\boldsymbol{{\mathbb{G}}_{\rm m}})}(\underline{V}^{\otimes r})=\Xi(\mathcal{H}_r),\]
\[{\rm End}_{\Xi(\mathcal{H}_r)}(\underline{V}^{\otimes r})=\mathbb{C}\Phi({\rm GL}_n\times\boldsymbol{{\mathbb{G}}_{\rm m}}),\]
where $\Phi:{\rm GL}_n\times\boldsymbol{{\mathbb{G}}_{\rm m}}\rightarrow {\rm GL}(\underline{V}^{\otimes r})$ and $\Xi:\mathcal{H}_r\rightarrow {\rm End}(\underline{V}^{\otimes r})$ are natural representations.
This result is called Levi Schur-Weyl duality. The quantum $q$-Schur duality is the $q$-deformation of classical Schur-Weyl duality. A natural question is the $q$-deformation of Levi Schur-Weyl duality, which is the purpose of this article.

The article is divided into three parts. In the first part, we introduce some basic concepts, especially quantum group $\mathbf{U}_q(\mathfrak{gl}_n)$ and Hecke algebra $\mathbf{H}(\mathfrak{S}_r)$. We describe their natural representations on $V^{\otimes r}$, where $V$ is an $n$-dimensional vector space. 
Furthermore, we recall the quantum $q$-Schur duality given in \cite{Jim86} and \cite{LW}. 
In the second part, we define the duplex Hecke algebra $\mathfrak{H}\kern -.6em\mathfrak{H}_r$. We give the representation of $\mathfrak{H}\kern -.6em\mathfrak{H}_r$ on $\underline{V}^{\otimes r}$ and prove it is well defined. 
The image of the representation is denoted  by $\mathcal{D}(n,r)$. 
In the third part, we define Levi quantum group $L_q(\mathfrak{gl}_{n+1})$ and prove Levi quantum group $L_q(\mathfrak{gl}_{n+1})$ and duplex Hecke algebra $\mathfrak{H}\kern -.6em\mathfrak{H}_r$ form double centralizes:
$${\rm End}_ {\mathcal{D}(n,r)}(\underline{V}^{\otimes r})=\Phi(L_q(\mathfrak{gl}_{n+1})),$$
$${\rm End}_{L_q(\mathfrak{gl}_{n+1})}(\underline{V}^{\otimes r})=\mathcal{D}(n,r).$$

In this article, $q$ is transcendental over $\mathbb{Q}$.

\section{Preliminaries}
We recall some fundamental results on the quantum $q$-Schur duality.
\begin{defn}\label{quantum group}
	The quantum group $\mathbf{U}_q(\mathfrak{sl}_n)$ is the associative algebra generated by $E_i,F_i,K_i,K_i^{-1} (1\leqslant i\leqslant n-1)$ over $\mathbb{Q}(q)$, which satisfies the following relations:
	\begin{equation}
		K_iK_i^{-1}=1=K_i^{-1}K_i;\quad
		K_iK_j=K_jK_i;
	\end{equation}
	\begin{align}
		K_iE_jK_i^{-1}=q^{c_{ij}}E_j;\quad K_iF_jK_i^{-1}=q^{-c_{ij}}F_j;
	\end{align}
	\begin{align}\label{sl1}
		E_iF_i-F_iE_i=\delta_{ij}\frac{K_i-K_i^{-1}}{q-q^{-1}};
	\end{align}
	\begin{align}
		\begin{split}
			\left.
			\begin{array}{ll}
				E_iE_j=E_jE_i\\
				F_iF_j=F_jF_i\\
			\end{array}
			\right\} \mbox{ if $c_{ij}=0$};
		\end{split}
	\end{align}
	\begin{align}\label{sl4}
		\begin{split}
			\left.
			\begin{array}{ll}
				E_i^2E_j-(q+q^{-1})E_iE_jE_i+E_jE_i^2=0\\
				F_i^2F_j-(q+q^{-1})F_iF_jF_i+F_jF_i^2=0\\
			\end{array}
			\right\} \mbox{ if $c_{ij}=-1$}.
		\end{split}
	\end{align}
	where $C=(c_{ij})_{(n-1)\times(n-1)}$ is the Cartan matrix of $\mathfrak{sl}_n$ .
\end{defn}

\begin{defn}\label{gl}
	The quantum group $\mathbf{U}_q(\mathfrak{gl}_n)$ is the associative algebra generated by $E_i,F_i,H_\alpha,H_\alpha^{-1} (1\leqslant i\leqslant n-1,1\leqslant\alpha\leqslant n)$ over $\mathbb{Q}(q)$, which satisfies the  relations (\ref{sl1}) - (\ref{sl4}) and the following relations:
	\begin{align}
		H_iH_i^{-1}=H_i^{-1}H_i=1,\quad
		H_iH_j=H_jH_i;
	\end{align}
	\begin{align}
		H_iE_iH_i^ {-1}=qE_i,\quad	H_iF_iH_i^ {-1}=q^{-1}F_i;
	\end{align}
	\begin{align}
		H_{i+1}E_iH_{i+1}^{-1}=q^{-1}E_i,\quad	H_{i+1}F_iH_{i+1}^{-1}=qF_i;
	\end{align}
	\begin{align}
		H_kE_iH_k^{-1}=E_i (k\neq i,i+1),\quad	H_kF_iH_k^{-1}=F_i (k\neq i,i+1).
	\end{align}
	\end{defn}

We have a natural embedding from $\mathbf{U}_q(\mathfrak{sl}_n)$ to $\mathbf{U}_q(\mathfrak{gl}_n)$ by $\mathbf{U}_q(\mathfrak{sl}_n)\rightarrow \mathbf{U}_q(\mathfrak{gl}_n):E_i\mapsto E_i,F_i\mapsto F_i,K_i\mapsto H_iH_{i+1}^{-1}$. Then we can regard $K_i,K_i^{-1} (1\leqslant i\leqslant n-1)$ as the elements in $\mathbf{U}_q(\mathfrak{gl}_n)$.
What's more, the quantum group $\mathbf{U}=\mathbf{U}_q(\mathfrak{gl}_n)$ is a Hopf algebra over $\mathbb{Q}(q)$
 with comultiplication $\Delta:\mathbf{U}\rightarrow \mathbf{U}\otimes\mathbf{U}$ such that
	$$H_i\mapsto H_i\otimes H_i;$$
	$$E_i\mapsto 1\otimes E_i+E_i\otimes K_i^{-1};$$
	$$F_i\mapsto F_i\otimes 1+K_i\otimes F_i.$$
For $\mathbf{U}$-modules $M$ and $N$, the tensor product $M\otimes N$ is again a $\mathbf{U}$-module by $u.(m\otimes n)=\Delta(u)(m\otimes n)=\sum u_1m\otimes u_2n\mbox{, where }\Delta(u)=\sum u_1\otimes u_2$.
Let $V$ be an $n$-dimensional vector space over $\mathbb{Q}(q)$. Let $\eta_i $ $(i=1,2,\cdots,n)$ be a basis of $V$. Note that $V$ is a natural module of $\mathbf{U}$ under the following actions:

	$$H_i\eta_i=q\eta_i,\quad H_i\eta_j=\eta_j\;\;(j\neq i);$$
     $$F_i\eta_i=\eta_{i+1},\quad F_i\eta_j=0\;\;(j\neq i);$$
	$$E_i\eta_{i+1}=\eta_i,\quad E_i\eta_j=0\;\;(j\neq i+1).$$

Obviously, $V^{\otimes r}$ is also a $\mathbf{U}$-module. For $u\in \mathbf{U}$, the action is under $\Delta^r\triangleq (\Delta\otimes 1^{r-2})\circ\cdots\circ(\Delta\otimes1)\circ\Delta(u)$, where 1 means the identity. We have
	$$\Delta^r(H_i)=\overbrace{H_i\otimes\cdots\otimes H_i}^{r};$$
	$$\Delta^r(E_i)=\sum\limits_{j=1}^{r}\overbrace{1\otimes\cdots\otimes1}^{j-1}\otimes{E_i}\otimes\overbrace{K_i^{-1}\otimes\cdots\otimes K_i^{-1}}^{r-j};$$
	$$\Delta^r(F_i)=\sum\limits_{j=1}^{r}\overbrace{K_i\otimes\cdots\otimes K_i}^{r-j}\otimes{F_i}\otimes\overbrace{1\otimes\cdots\otimes1}^{j-1}.$$

\begin{defn}
	The Hecke algebra $\mathbf{H}=\mathbf{H}_q(\mathfrak{S}_r) (\mathfrak{S}_r$ is the symmetric group generated by $s_i=(i,i+1)$ with $i=1,2,\cdots r-1)$ is a $\mathbb{Q}(q)$-algebra with unit $1$, generated by $T_{s_i}$ $(1\leq i\leq r-1)$ subject to the following relations:
		$$(T_{s_i}+q)(T_{s_i}-q^{-1})=0;$$
		$$T_{s_i}T_{s_{i+1}}T_{s_i}=T_{s_{i+1}}T_{s_i}T_{s_{i+1}};$$
		$$T_{s_i}T_{s_j}=T_{s_j}T_{s_i},(|i-j|>1).$$
	Define $T_\sigma=T_{i_1}\cdots T_{i_k}$, for any reduced expression $\sigma=s_{i_1}\cdots s_{i_k}\in\mathfrak{S}_r$. (It is independent of choices of reduced expression.)
\end{defn}

We denote $\mathbf{H}_q(\mathfrak{S}_r)$ simply by $\mathbf{H}(\mathfrak{S}_r)$.
A right action on a basis $\{M_f=\eta_{f(1)}\otimes \eta_{f(2)}\otimes\cdots\otimes \eta_{f(r)}\mid f:\{1,2\cdots ,r\}\mapsto\{1,2\cdots ,n\}\}$ of $V^{\otimes r}$  is given by:
\begin{align*}
	\begin{split}
		M_fT_{s_i}=\left\{
		\begin{array}{ll}
			M_{f{s_i}},&  f(i)<f(i+1),\\
			M_{f{s_i}}+(q^{-1}-q)M_f,&  f(i)>f(i+1),\\
			q^{-1}M_f,&  f(i)=f(i+1).
		\end{array}
		\right .
	\end{split}
\end{align*}
where $f{s_i}$ is defined by $(f{s_i})(j)=f(s_i(j))$.

The left action of $\mathbf{U}$ and the right action of $\mathbf{H}$ on $V^{\otimes r}$ are denoted by $\Phi$ and $\Psi$ respectively.
\begin{defn}
Suppose $r$ and $n$ are two positive integers.
	\begin{itemize}
	\item[(1)] 
    Let $\Lambda(n,r)=\{(\lambda_1,\dots,\lambda_n)\in\mathbb{N}^n\mid \sum\limits_{i=1}^n\lambda_i=r\}$ be the set of compositions of $r$ into $n$ parts.
	\item[(2)] 
	Let
	$P(n,r)=\{(\lambda_1,\lambda_2\cdots\lambda_n)\in\mathbb{N}^n\mid\sum\limits_{i=1}^n\lambda_i=r,\lambda_1\geq\cdots\geq\lambda_n\geq 0\}$
	be the set of partitions of $r$ of length not exceeding $n$.
\end{itemize}	 
\end{defn}

\begin{thm}\label{q duality} (quantum q-Schur duality)
	\begin{itemize}
	\item[(1)] The actions of $\mathbf{U}$ and $\mathbf{H}$ commute and they form double centralizers, i.e:
	$$\mathbf{U}=\mathbf{U}_q(\mathfrak{gl}_n)\overset{\Phi}{\rightarrow}V^{\otimes r}\overset{\Psi}{\leftarrow}\mathbf{H}_q(\mathfrak{S}_r)=\mathbf{H},$$
	$${\rm End}_\mathbf{H}(V^{\otimes r})=\Phi(\mathbf{U}),$$
	$${\rm End}_\mathbf{U}(V^{\otimes r})=\Psi(\mathbf{H}).$$

	\item[(2)] As a $(\mathbf{U},\mathbf{H})$-module, 
	$$V^{\otimes r}\simeq\bigoplus\limits_{\lambda\in P(n,r)} L_q(\lambda)\otimes S_q(\lambda),$$ where $L_q(\lambda)$ is the irreducible $\mathbf{U}$-module with highest weight $\lambda$ and $S_q^\lambda$ is the irreducible $\mathbf{H}$-module called the Specht module (see \cite{DDPW} or \cite{JAN95} for details). 
\end{itemize}

\end{thm}

This theorem is due to Jimbo (see \cite{Jim86}), which can be proved by $R$-Matrix. Furthermore, there are different approaches to the $q$-Schur duality using the $q$-Schur algebra developed in \cite{DJ89} and \cite{BLM}.

\begin{defn}
	Let $\lambda=(\lambda_1,\cdots,\lambda_m)\in\Lambda(m,n)$. For $1\leq i\leq m$, let
	$$[a,b]:=[a,a+1,\cdots,b-1,b]  ,(a<b);$$
	$$R_i^{\lambda}=[\lambda_1+\cdots+\lambda_{i-1}+1,\lambda_1+\cdots+\lambda_i].$$
	
	If $\lambda_i=0$, then $R_i^{\lambda}:=\emptyset$, In this way, we get a decomposition of $[1,n]$ into a disjoint union of subsets:
	\[[1,n]=R_1^{\lambda}\cup R_2^{\lambda}\cup\cdots\cup R_m^{\lambda}.\]
	Let $\mathfrak{S}_{\lambda}$ be the subgroup of $\mathfrak{S}_n$ stabilizing this decomposition. More precisely, define
	\[\mathfrak{S}_\lambda:=\{w\in \mathfrak{S}_n\mid wR_i^{\lambda}=R_i^{\lambda}, 1\leq i\leq m\}.\]
	We call $\mathfrak{S}_\lambda$ the Young subgroup of $\mathfrak{S}_n$ defined by the composition $\lambda$.
\end{defn}
From the above definition, $\prescript{\lambda}{}{\mathfrak{S}}$ is the set of shortest left coset representatives of $\mathfrak{S}_\lambda$ in $\mathfrak{S}_n$. Let $\prescript{\lambda}{}{\mathfrak{S}}^{\mu}$ be the set of shortest $(\mathfrak{S}_\lambda,\mathfrak{S}_\mu)$-double coset representatives. These notations will be used in section 4.

\section{Doubled Hecke algebras}
In this section, we will introduce Doubled Hecke algebra. Let $\mathscr{A}=\mathbb{Q}(q)$.  
For given positive integers $r$ and $l$ with $r>l$, we consider the following $\mathscr{A}$-algebra $\mathfrak{H}\kern -.6em\mathfrak{H}^l_r$ with unit 1, defined by generators $\{{\rm T}_{s_i}\mid i=1,2,\cdots,r-1\}\cup\{x_\sigma\mid \sigma\in \mathfrak{S}_l\}$ subject to the following relations:
\begin{align}\label{qha1}
	({\rm T}_{s_i}+q)({\rm T}_{s_i}-q^{-1})=0\;\;(0< i\leq r-1);
\end{align}
\begin{align}\label{qha2}
	{\rm T}_{s_i}{\rm T}_{s_j}={\rm T}_{s_j}{\rm T}_{s_i}\;\;(0< i\ne j\leq r-1,|i-j|>1);
\end{align}
\begin{align}\label{qha3}
	{\rm T}_{s_i}{\rm T}_{s_j}{\rm T}_{s_i}={\rm T}_{s_j}{\rm T}_{s_i}{\rm T}_{s_j}\;\;(0< i\ne j\leq r-1,|i-j|=1);
\end{align}
\begin{align}\label{qha4}
	\begin{split}
		x_\sigma x_{s_i}=\left\{
		\begin{array}{ll}
			x_{\sigma s_i},&  l(\sigma s_i)=l(\sigma)+1,\\
			x_{\sigma s_i}+(q^{-1}-q)x_\sigma,&  l(\sigma s_i)=l(\sigma)-1;
		\end{array}
		\right.
	\end{split}
\end{align}
\begin{align}\label{qha5}
	\begin{split}
		x_{s_i} x_\sigma=\left\{
		\begin{array}{ll}
			x_{s_i\sigma},&  l(s_i\sigma)=l(\sigma)+1,\\
			x_{s_i\sigma}+(q^{-1}-q)x_\sigma,&  l(s_i\sigma)=l(\sigma)-1;
		\end{array}
		\right.
	\end{split}
\end{align}
\begin{align}\label{qha6}
	\begin{split}
		{\rm T}_{s_i} x_\sigma=\left\{
		\begin{array}{ll}
			x_{s_i\sigma},&  l(s_i\sigma)=l(\sigma)+1, i<l,\\
			x_{s_i\sigma}+(q^{-1}-q)x_\sigma,&  l(s_i\sigma)=l(\sigma)-1, i<l;
		\end{array}
		\right.
	\end{split}
\end{align}
\begin{align}\label{qha7}
	\begin{split}
		x_\sigma {\rm T}_{s_i}=\left\{
		\begin{array}{ll}
			x_{\sigma s_i},&  l(\sigma s_i)=l(\sigma)+1, i<l,\\
			x_{\sigma s_i}+(q^{-1}-q)x_\sigma,&  l(\sigma s_i)=l(\sigma)-1, i<l;
		\end{array}
		\right.
	\end{split}
\end{align}
\begin{align}\label{qha8}
	{\rm T}_{s_i}x_\sigma=q^{-1}x_\sigma=x_\sigma {\rm T}_{s_i},  i>l.
\end{align}
This is an infinite-dimensional associative algebra. We call $\mathfrak{H}\kern -.6em\mathfrak{H}^l_r$ the $l$-{th} duplex Hecke algebra of $\mathbf{H}(\mathfrak{S}_r)$. By definition, the subalgebra generated by $\{{\rm T}_{s_i}\mid i=1,2,\cdots,r-1\}$ is isomorphic to $\mathbf{H}(\mathfrak{S}_r)$.

Furthermore, we make an appointment that $\mathfrak{H}\kern -.6em\mathfrak{H}^0_r:=\mathbf{H}(\mathfrak{S}_r)$ and $\mathfrak{H}\kern -.6em\mathfrak{H}^r_r:=\langle {\rm T}_{s_i},x_\sigma\rangle$ with all ${\rm T}_{s_i}, x_\sigma$, $i=1,\ldots,r-1$ and $\sigma\in \mathfrak{S}_r$ satisfying (\ref{qha1}) - (\ref{qha7}). Now we combine all $\mathfrak{H}\kern -.6em\mathfrak{H}_r^l (l=0,1,\cdots,r)$.

\begin{defn}
	The duplex Hecke algebra $\mathfrak{H}\kern -.6em\mathfrak{H}_r$ of $\mathbf{H}(\mathfrak{S}_r)$ is a $\mathbb{Q}(q)$-associative algebra with unit $1$ and generators ${\rm T}_{s_i}$ $(i=1,\cdots,r-1)$, $x^{(l)}_\sigma$ for $\sigma\in\mathfrak{S}_l$ $(l=1,2,\cdots,r)$, and with relations as (\ref{qha1}) - (\ref{qha8}) in which $x_\sigma,x_{s_i}$ are replaced by $x_\sigma^{(l)},x_{s_i}^{(l)}$ and addition ones:
	\begin{align}\label{qha9}
		x_\sigma^{(l)}x_\gamma^{(k)}=0 ,\mbox{ for }\sigma\in\mathfrak{S}_l,\gamma\in\mathfrak{S}_k,k\neq l.
	\end{align}
\end{defn}
Naturally, $\mathbf{H}(\mathfrak{S}_r)$ is a subalgebra of $\mathfrak{H}\kern -.6em\mathfrak{H}_r$.

Let $\underline{V}$ be a one-dimesional extension of $V$ with enhanced vector  $\eta_{n+1}$. 
Then $\boldsymbol{\eta_i}=\eta_ {i_1}\otimes\cdots\otimes\eta_{i_r}$, where  $\boldsymbol{i}=(i_1,\cdots,i_r)\in\mathcal{N}^r$
 $(\mathcal{N}:=\{1,2,\cdots,n+1\})$, form a basis of $\underline{V}^{\otimes r}$. For a given $\boldsymbol{j}=(j_1,\cdots,j_r)\in\mathcal{N}^r$, we call 
 the number of elements $j_i$ which is not equal to $n+1$ the $\underline{n}$-rank of $\boldsymbol{j}$, denoted by $rk_{\underline{n}}(\boldsymbol{j})$.
  All vectors with $\underline{n}$-rank equal to $l$ form a set $\mathcal{N}^r_l$ and it is easy to check $\mathcal{N}^r=\bigcup_{l=0}^r \mathcal{N}^r_l$.
Then $\underline{V}^{\otimes r}$ can be decomposed into the direct sum of $\mathscr{A}$-subspace: $\underline{V}^{\otimes r}=\bigoplus_{l=0}^r\underline{V}^{\otimes r}_l$, where 
$\underline{V}^{\otimes r}_l:=\sum_{\boldsymbol{i}\in \mathcal{N}^r_l}\mathscr{A}\boldsymbol{\eta_{i}}$, $l=0,1,\cdots,r$.
    Naturally, each $\underline{V}^{\otimes r}_l$ is stable under the natural action of $\mathbf{H}(\mathfrak{S}_r)$, i.e. 
     $\mathbf{H}(\mathfrak{S}_r).\underline{V}^{\otimes r}_l\subseteq\underline{V}^{\otimes r}_l$. We can get a representation of $\mathbf{H}(\mathfrak{S}_r)$ over $\underline{V}^{\otimes r}_l$, denoted by $\Psi|_l$.

   Consider $I=\{i_1,i_2,\cdots,i_l\} (i_1<i_2\cdots<i_l,1\leq i_j\leq r,i_j\in\mathbb{Z},1\leq j \leq l)$. 
   We denote $\underline{V}_I^{\otimes r}$ as the subspace spanned by $\{\eta_{j_1}\otimes\cdots\otimes\eta_{j_r}\mid 1\leq j_k\leq n,k\in I;j_k=n+1, k\notin I\}$.
   In particular, $\underline{V}^{\otimes r}_{\underline{l}}=V^{\otimes l}\otimes \eta_{n+1}^{\otimes r-l}$, where $\underline{l}=\{1,2,\cdots,l\}$.
 
 \begin{lem}
     Let $I$ be the set above. For any $\boldsymbol{\eta_j}\in\underline{V}^{\otimes r}_I$, there exist $\omega_I\in\mathbf{H}(\mathfrak{S}_r)$ and $\boldsymbol{\eta_{j'}}\in\underline{V}^{\otimes r}_{\underline{l}}$
     such that $\Psi(\omega_I)\boldsymbol{\eta_j}=\boldsymbol{\eta_{j'}}$, where $\Psi$ is the natural representation of $\mathbf{H}(\mathfrak{S}_r)$ on $\underline{V}^{\otimes r}$.
 \end{lem}
 \begin{proof}  
    Let
    $\boldsymbol{\eta_j}=\eta_{n+1}\otimes\cdots\otimes\eta_{j_1}\otimes\eta_{n+1}\otimes\cdots\otimes\eta_{j_l}\otimes\cdots\otimes\eta_{n+1}\in\underline{V}^{\otimes r}_I$, where the  ${i_m}$-{th} position of $\boldsymbol{\eta_j}$ is $\eta_{j_m}$ with $1\leq j_m\leq n$ and other positions of $\boldsymbol{\eta_j}$ is $\eta_{n+1}$. Let $\boldsymbol{\eta_{j'}}=\eta_{j_1}\otimes\eta_{j_2}\cdots\otimes\eta_{j_l}\otimes{\eta_{n+1}\otimes\cdots\otimes\eta_{n+1}}\in\underline{V}^{\otimes r}_{\underline{l}}$.       
    Notice $T_{s_i}\in\mathbf{H}(\mathfrak{S}_r)$ is invertible, it is enough to prove there exists $\omega_I'\in\mathbf{H}(\mathfrak{S}_r)$ such that $\Psi(\omega_I')\boldsymbol{\eta_{j'}}=\boldsymbol{\eta_{j}}$.
    
    We move the $l$-{th} element $\eta_{j_l}$ of $\boldsymbol{\eta_{j'}}$ to ${i_l}$-{th} position at first. Since $j_{l}<n+1$, we have  
  $$T_{s_l}.\boldsymbol{\eta_{j'}}=\eta_{j_1}\otimes\eta_{j_2}\otimes\cdots\otimes\eta_{j_{l-1}}\otimes\eta_{n+1}\otimes\eta_{j_l}\overbrace{\otimes{\eta_{n+1}\otimes\cdots\otimes\eta_{n+1}}}^{r-l-1}.$$
      Then we apply the actions of $T_{s_{l+1}},\cdots,T_{s_{i_{l}-1}}$ by turns and get an element $$\eta_{j_1}\otimes\eta_{j_2}\otimes\cdots\otimes\eta_{j_{l-1}}\otimes\eta_{n+1}\otimes\cdots\otimes{\eta_{j_l}}\overbrace{\otimes{\eta_{n+1}\otimes\cdots\otimes\eta_{n+1}}}^{r-i_l},$$
    where $\eta_{j_l}$ is at the $i_l$-th position.
    Similarly, we can move $\eta_{l-1}$ to $i_{l-1}$-th position, since $j_{l-1}<n+1$ and $i_{l-1}<i_l$. Repeating the process above, we finally get an element $\omega_I'\in\mathbf{H}(\mathfrak{S}_r)$ which transforms $\boldsymbol{\eta_j}$ to $\boldsymbol{\eta_{j'}}$ as desired.    
 \end{proof}
  
 We now consider the representation of $\mathfrak{H}\kern -.6em\mathfrak{H}_r$ on $\underline{V}^{\otimes r}$.
   There is a natural representation of $\mathbf{H}(\mathfrak{S}_l)$ on $V^{\otimes l}$, denoted by $\Psi^V_l$. It transforms $\eta_1\otimes\cdots\otimes \eta_l$ to $T_\sigma.(\eta_1\otimes\cdots\otimes \eta_l)$, where $\sigma\in \mathfrak{S}_l$. Recall the notation $\underline{V}^{\otimes r}_{\underline{l}}=V^{\otimes l}\otimes\eta_{n+1}^{\otimes r-l}$. We extend $\Psi^V_l$ and define linear operator as follow:
\begin{align*}
	\psi_\sigma=\Psi^V_l(T_\sigma)\otimes {\rm id}^{\otimes r-l}\in {\rm End}_{\mathscr{A}}(\underline{V}^{\otimes r}_{\underline{l}}).
\end{align*}
We can extend $\psi_\sigma$ to an element $\psi_\sigma^{\underline{l}}$ of ${\rm End}_{\mathscr{A}}(\underline{V}^{\otimes r}_l)$ by annihilating other $\underline{V}^{\otimes r}_I$ with $I\neq\underline{l}$. Note that $\mathbf{H}(\mathfrak{S}_r)$ is a subalgebra of  $\mathfrak{H}\kern -.6em\mathfrak{H}_r$. Therefore we can define the action of ${\rm T}_{s_i}\in \mathfrak{H}\kern -.6em\mathfrak{H}_r^l$ on ${\rm End}_\mathscr{A}(\underline{V}^{\otimes r}_l)$, which is compatible with the action of $\Psi|_l(T_{s_i})$.

In general, for $\boldsymbol{\eta_j}\in\underline{V}^{\otimes r}_I$ with $\#I=l$, we have $\boldsymbol{\eta_j}=\Psi(\omega_I^{-1})\boldsymbol{\eta_{j'}}$ for some $\boldsymbol{\eta_{j'}}\in\underline{V}^{\otimes r}_{\underline l} $. 
Then  $\Psi(\omega_I^{-1})\circ \psi_\sigma\circ\Psi(\omega_I)$ lies in ${\rm End}_\mathscr{A}(\underline{V}^{\otimes r}_I)$ for any $\sigma\in \mathfrak{S}_l$, which extend to an element of ${\rm End}_\mathscr{A}(\underline{V}^{\otimes r}_l)$ and denote it by $\psi_\sigma^I$. \par
Before defining a representation of $\mathfrak{H}\kern -.6em\mathfrak{H}_r$ on $\underline{V}^{\otimes r}$, we give a remark. Let $\Phi'$ be the natural representation of $\mathbf{U}_q(\mathfrak{gl}_{n})$ on ${V}^{\otimes r}$ and $\Phi$ be the natural representation of $\mathbf{U}_q(\mathfrak{gl}_{n+1})$ on $\underline{V}^{\otimes r}$. There is a natural embedding $\Phi'(\mathbf{U}_q(\mathfrak{gl}_n))\hookrightarrow \Phi(\mathbf{U}_q(\mathfrak{gl}_{n+1}))$, such that each element $g$ in $\mathbf{U}_q(\mathfrak{gl}_{n})$ acts as identity on  $\eta_{n+1}$. We still denote as $\Phi(g)$ for simplicity.

\begin{lem}\label{l rep of qha}
	The following statements hold.
	\begin{itemize}
		\item[(1)]For $1\leq l\leq r$, there is an algebra homomorphism $\Xi_l: \mathfrak{H}\kern -.6em\mathfrak{H}^l_r\rightarrow {\rm End}_\mathscr{A}(\underline{V}^{\otimes r}_l)$ defined by sending ${\rm T}_{s_i}\mapsto \Psi|_l(T_{s_i})$ and $x_\sigma\mapsto \psi_\sigma^{\underline{l}}$.
		\item[(2)]For $l=0$, there is an algebra homomorphism $\Xi_0: \mathfrak{H}\kern -.6em\mathfrak{H}^0_r\rightarrow {\rm End}_\mathscr{A}(\underline{V}^{\otimes r}_l)$ defined by sending ${\rm T}_{s_i}\mapsto \Psi|_l(T_{s_i})=q^{-1}{\rm id}$.		
		\item[(3)] For any $l\in\{0,1,\cdots,r\}$ and any $g\in\mathbf{U}_q(\mathfrak{gl}_n)\hookrightarrow \mathbf{U}_q(\mathfrak{gl}_{n+1})$, $\Phi(g)$ commutes with any elements from $\Xi_l(\mathfrak{H}\kern -.6em\mathfrak{H}^l_r)$ in ${\rm End}_\mathscr{A}(\underline{V}^{\otimes r}_l)$.		
		\item[(4)] On $\underline{V}^{\otimes r}$, there is a representation $\Xi$ of $\mathfrak{H}\kern -.6em\mathfrak{H}_r$ defined via:
		\begin{itemize}
			\item[(4.1)] $\Xi|_{{\rm T}_{s_i}}=\Psi$, where $\Xi|_{{\rm T}_{s_i}}$ means all action of ${\rm T}_{s_i}$ over $\Xi$;
			
			\item[(4.2)] For any $x_\sigma\in \mathfrak{H}\kern -.6em\mathfrak{H}^l_r$, $l=1,2,\cdots,r$, $\Xi(x_\sigma)|_{\underline{V}^{\otimes r}_l}=\Xi_l(x_\sigma)$ and $\Xi(x_\sigma)|_{\underline{V}^{\otimes r}_k}=0$ for $k\ne l$.
		\end{itemize}
	\end{itemize}
\end{lem}

\begin{proof}
	(1) For $1\leq l\leq r$, we need to show that $\Xi_l$ keeps the relations (\ref{qha1}) - (\ref{qha8}).		
	Recall the action of Hecke algebra on $\underline{V}^{\otimes r}$, it is easy to get the following relations:	
	\begin{align}\label{check1}
		\Xi_l({\rm T}_{s_i})^2=\Xi_l(1+(q^{-1}-q){\rm T}_{s_i}),\quad \Xi_l({\rm T}_{s_i})\Xi_l({\rm T}_{s_j})=\Xi_l({\rm T}_{s_j})\Xi_l({\rm T}_{s_i}),
	\end{align}
	with $0\leq i\ne j\leq r-1, |j-i|>1$, and
	\begin{align}\label{check2}
		\Xi_l({\rm T}_{s_i})\Xi_l({\rm T}_{s_j})\Xi_l({\rm T}_{s_i})=\Xi_l({\rm T}_{s_j})\Xi_l({\rm T}_{s_i})\Xi_l({\rm T}_{s_j}),
	\end{align}
	with $0\leq i\ne j\leq r-1, |j-i|=1.$
		
	Suppose that $\boldsymbol{\eta_j}\in\underline{V}^{\otimes r}_{\underline{l}}$,  $\boldsymbol{\eta_k}\in\underline{V}^{\otimes r}_{l}$ with $\boldsymbol{\eta_k}\notin\underline{V}^{\otimes r}_{\underline{l}}$. Note that $\psi_\sigma^{\underline{l}}|_{\underline{V}^{\otimes r}_{\underline{l}}}=\psi_\sigma$ and $\underline{V}^{\otimes r}_{\underline{l}}\simeq V^{\otimes l}$. The action of $x_\sigma$ comes from the action of $T_\sigma\in\mathbf{H}(\mathfrak{S}_l)$,
    which shows
	\[\begin{split}
		\Xi_l(x_\sigma)\Xi_l( x_{s_i})(\boldsymbol{\eta_j})=\left\{
		\begin{array}{ll}
			\Xi_l(x_{\sigma s_i})(\boldsymbol{\eta_j}),& l(\sigma s_i)=l(\sigma)+1,\\
			\Xi_l(x_{\sigma s_i}+(q^{-1}-q)x_\sigma)(\boldsymbol{\eta_j}),&  l(\sigma s_i)=l(\sigma)-1.
		\end{array}
		\right.
	\end{split}\]
	Moreover, we have
	\[\begin{split}
		\Xi_l(x_\sigma)\Xi_l( x_{s_i})(\boldsymbol{\eta_k})=0=\left\{
		\begin{array}{ll}
			\Xi_l(x_{\sigma s_i})(\boldsymbol{\eta_k}),&  l(\sigma s_i)=l(\sigma)+1,\\
			\Xi_l(x_{\sigma s_i}+(q^{-1}-q)x_\sigma)(\boldsymbol{\eta_k}),&  l(\sigma s_i)=l(\sigma)-1.
		\end{array}
		\right.
	\end{split}\]
	Hence we get
	\begin{align}\label{check3}
		\begin{split}
			\Xi_l(x_\sigma)\Xi_l( x_{s_i})=\left\{
			\begin{array}{ll}
				\Xi_l(x_{\sigma s_i}),&  l(\sigma s_i)=l(\sigma)+1,\\
				\Xi_l(x_{\sigma s_i}+(q^{-1}-q)x_\sigma),&  l(\sigma s_i)=l(\sigma)-1.
			\end{array}
			\right.
		\end{split}
	\end{align}
	In the same way, we can get
	\begin{align}\label{check4}
		\begin{split}
			\Xi_l(x_{s_i})\Xi_l( x_\sigma)=\left\{
			\begin{array}{ll}
				\Xi_l(x_{s_i\sigma}),&  l(s_i\sigma)=l(\sigma)+1,\\
				\Xi_l(x_{s_i\sigma}+(q^{-1}-q)x_\sigma),&  l(s_i\sigma)=l(\sigma)-1.
			\end{array}
			\right.
		\end{split}
	\end{align}
	
	When $i<l$, note that $\Psi|_l(T_{s_i})$ is equal to $\psi_{s_i}^{\underline{l}}$ by  definition. Using this one checks that
	\[\begin{split}
		\Xi_l({\rm T}_{s_i})\Xi_l( x_\sigma)(\boldsymbol{\eta_j})=\left\{
		\begin{array}{ll}
			\Xi_l(x_{s_i\sigma})(\boldsymbol{\eta_j}),&  l(s_i\sigma)=l(\sigma)+1, i<l,\\
			\Xi_l(x_{s_i\sigma}+(q^{-1}-q)x_\sigma)(\boldsymbol{\eta_j}),&  l(s_i\sigma)=l(\sigma)-1, i<l,
		\end{array}
		\right.
	\end{split}\]
	and\[\begin{split}
		\Xi_l(x_\sigma)\Xi_l ({\rm T}_{s_i})(\boldsymbol{\eta_j})=\left\{
		\begin{array}{ll}
			\Xi_l(x_{\sigma s_i})(\boldsymbol{\eta_j}),&  l(\sigma s_i)=l(\sigma)+1, i<l,\\
			\Xi_l(x_{\sigma s_i}+(q^{-1}-q)x_\sigma)(\boldsymbol{\eta_j}),&  l(\sigma s_i)=l(\sigma)-1, i<l.
		\end{array}
		\right.
	\end{split}\]
	Since the action of $\Xi_l(T_{s_i})=\Psi|_l(T_{s_i})$ on $\boldsymbol{\eta_k}$ does not change the latter $r-l$ components of $\boldsymbol{\eta_k}$, we have $\Xi_l({\rm T}_{s_i})(\boldsymbol{\eta_k})\notin\underline{V}^{\otimes r}_{\underline{l}}$. It is clear that
	\[\begin{split}
		\Xi_l({\rm T}_{s_i})\Xi_l( x_\sigma)(\boldsymbol{\eta_k})=0=\left\{
		\begin{array}{ll}
			\Xi_l(x_{s_i\sigma})(\boldsymbol{\eta_k}),&  l(s_i\sigma)=l(\sigma)+1, i<l,\\
			\Xi_l(x_{s_i\sigma}+(q^{-1}-q)x_\sigma)(\boldsymbol{\eta_k}),&  l(s_i\sigma)=l(\sigma)-1, i<l,
		\end{array}
		\right.
	\end{split}\]
	and
	\[\begin{split}
		\Xi_l(x_\sigma)\Xi_l ({\rm T}_{s_i})(\boldsymbol{\eta_k})=0=\left\{
		\begin{array}{ll}
			\Xi_l(x_{\sigma s_i})(\boldsymbol{\eta_k}),&  l(\sigma s_i)=l(\sigma)+1, i<l,\\
			\Xi_l(x_{\sigma s_i}+(q^{-1}-q)x_\sigma)(\boldsymbol{\eta_k}),&  l(\sigma s_i)=l(\sigma)-1, i<l.
		\end{array}
		\right.
	\end{split}
	\]	
     Hence we get
	\begin{align}\label{check5}
		\begin{split}
			\Xi_l({\rm T}_{s_i})\Xi_l( x_\sigma)=\left\{
			\begin{array}{ll}
				\Xi_l(x_{s_i\sigma}),&  l(s_i\sigma)=l(\sigma)+1, i<l,\\
				\Xi_l(x_{s_i\sigma}+(q^{-1}-q)x_\sigma),&  l(s_i\sigma)=l(\sigma)-1, i<l,
			\end{array}
			\right.
		\end{split}
	\end{align}
	and
	\begin{align}\label{check6}
		\begin{split}
			\Xi_l(x_\sigma)\Xi_l( {\rm T}_{s_i})=\left\{
			\begin{array}{ll}
				\Xi_l(x_{\sigma s_i}),&  l(\sigma s_i)=l(\sigma)+1, i<l,\\
				\Xi_l(x_{\sigma s_i}+(q^{-1}-q)x_\sigma),&  l(\sigma s_i)=l(\sigma)-1, i<l.
			\end{array}
			\right.
		\end{split}
	\end{align}

     When $i>l$. The following formula follows from the definition:
	\begin{align}\label{check7}
		\Xi_l(x_{\sigma})\Xi_l({\rm T}_{s_i})=q^{-1}\Xi_l(x_{\sigma})=\Xi_l({\rm T}_{s_i})\Xi_l(x_{\sigma}), i>l.
	\end{align}
     
	It follows form (\ref{check1}) - (\ref{check7}) that $\Xi_l$ is an algebra homomorphism from $\mathfrak{H}\kern -.6em\mathfrak{H}^l_r$ to ${\rm End}_\mathscr{A}(\underline{V}^{\otimes r}_l)$.
	
	(2) In this situation, $\Xi_0$ obviously becomes an algebra homomorphism from $\mathfrak{H}\kern -.6em\mathfrak{H}^0_r$ to ${\rm End}_\mathscr{A}(\underline{V}^{\otimes r}_0)$, since $\mathfrak{H}\kern -.6em\mathfrak{H}^0_r=\mathbf{H}(\mathfrak{S}_r)$.
	
	(3) When $l=0$, it is just quantum $q$-Schur duality, since $\mathfrak{H}\kern -.6em\mathfrak{H}^0_r=\mathbf{H}(\mathfrak{S}_r)$. Now we suppose $1\leq l\leq r$. Note that the action of ${\rm T}_{s_i}$ is given by that of $T_{s_i}$ in  $\mathbf{H}(\mathfrak{S}_r)$. It is easy to check that $\Phi(g)\Xi_l({\rm T}_{s_i})=\Xi_l({\rm T}_{s_i})\Phi(g)$ for any $g\in \mathbf{U}_q(\mathfrak{gl}_n)$ because of the quantum $q$-Schur duality. 
	
	 Recall that the generators of $\mathbf{U}_q(\mathfrak{gl}_n)$ are $E_i,F_i,H_j,H_j^{-1}(1\leq i\leq n-1,1\leq j\leq n)$. 
	 It is clear from the definition that
	$\Delta^r(H_i)\underline{V}^{\otimes r}_I\subseteq\underline{V}^{\otimes r}_I$. 
	Let $\boldsymbol{l}=(l_1,\cdots,l_r)\in \mathcal{N}^r_l$ and $I\subset\underline{r}$ with $\#I=l$.
	If $l_j\neq i+1$, then
	$$\overbrace{1\otimes\cdots\otimes1} ^{j-1}\otimes{E_i}\otimes\overbrace{K_i^{-1}\otimes\cdots\otimes K_i^{-1}}^{r-j}$$
	acts on $\boldsymbol{\eta_l}$ as zero.
	Otherwise,
	\begin{align*}
	&\overbrace{1\otimes\cdots\otimes1} ^{j-1}\otimes{E_i}\otimes\overbrace{K_i^{-1}\otimes\cdots\otimes K_i^{-1}}^{r-j}(\boldsymbol{\eta_l})\\
	=&\eta_{l_1}\otimes\eta_{l_2}\cdots\otimes\eta_{l_{j-1}}\otimes \eta_i\otimes K_i^{-1}\eta_{l_{j+1}}\cdots\otimes K_i^{-1}\eta_{l_r}.
	\end{align*}
    It shows
	\[\overbrace{1\otimes\cdots\otimes1} ^{j-1}\otimes{E_i}\otimes\overbrace{K_i^{-1}\otimes\cdots\otimes K_i^{-1}}^{r-j}\underline{V}^{\otimes r}_I\subseteq\underline{V}^{\otimes r}_I.\]
	Hence we have $\Delta^r(E_i)\underline{V}^{\otimes r}_I\subseteq\underline{V}^{\otimes r}_I$. In the same way, we can get  $\Delta^r(F_i)\underline{V}^{\otimes r}_I\subseteq\underline{V}^{\otimes r}_I$. 
	We now deduce that $\Phi(g)(\underline{V}^{\otimes r}_{I})\subseteq\underline{V}^{\otimes r}_{I}$ for any $g\in \mathbf{U}_q(\mathfrak{gl}_n)$. Hence we have $$\Phi(g)\Xi_l(x_\sigma)(\boldsymbol{\eta_k})=0=\Xi_l(x_\sigma)\Phi(g)(\boldsymbol{\eta_k})$$ for any $\boldsymbol{\eta_k}\notin\underline{V}^{\otimes r}_{\underline{l}}.$

	Since 
    $x_\sigma$ has the same action as $T_\sigma$ in $\mathbf{H}(\mathfrak{S}_l)$ on $\underline{V}^{\otimes r}_{{l}}$. According to quantum $q$-Schur duality, we have $\Xi_l(x_\sigma)\Phi(g)(\boldsymbol{\eta_j})=\Phi(g)\Xi_l(x_\sigma)(\boldsymbol{\eta_j})$ for any $\boldsymbol{\eta_j}\in\underline{V}^{\otimes r}_{\underline{l}}$. Thus $\Xi_l(\mathfrak{H}\kern -.6em\mathfrak{H}^l_r)$ commutes with $\Phi(g)$ for any $g\in \mathbf{U}_q(\mathfrak{gl}_n)$in ${\rm End}_\mathscr{A}(\underline{V}^{\otimes r}_l)$.
	
	(4) Since $\Psi$ is the representation of $\mathbf{H}(\mathfrak{S}_r)$, $\Xi$ keeps the relations (\ref{qha1}) - (\ref{qha3}). Moreover, since $\Xi|_{\underline{V}^{\otimes r}_l}=\Xi_l$ for $1\leq l\leq r$, we have:
	\begin{align*}
		\Xi(x_{\sigma}^{(l)})\circ\Xi(x_{\mu}^{(k)})= \Xi_l(x_{\sigma}^{(l)})\Xi_k(x_{\mu}^{(k)})=0 \mbox{ for  } \mu\in\mathfrak{S}_k, \sigma\in\mathfrak{S}_l, k,l\in\underline{r}, k\neq l;
	\end{align*}
	\begin{align*}
		\begin{split}
			\Xi({\rm T}_{s_i})\circ\Xi(x_{\sigma}^{(l)})=\Psi(T_{s_i})\Xi_l(x_{\sigma}^{(l)})=\left\{
			\begin{array}{ll}
				\Xi_l(x_{s_i\sigma}^{(l)}),&  l(s_i\sigma)=l(\sigma)+1, i<l,\\
				\Xi_l(x_{s_i\sigma}^{(l)}+(q^{-1}-q)x_\sigma^{(l)}),&  l(s_i\sigma)=l(\sigma)-1, i<l;
			\end{array}
			\right.
		\end{split}
	\end{align*}	
	\begin{align*}
		\begin{split}
			\Xi(x_{\sigma}^{(l)})\circ\Xi({\rm T}_{s_i})=\Xi_l(x_{\sigma}^{(l)})\Psi(T_{s_i})=\left\{
			\begin{array}{ll}
				\Xi_l(x_{\sigma s_i}^{(l)}),&  l(\sigma s_i)=l(\sigma)+1, i<l,\\
				\Xi_l(x_{\sigma s_i}^{(l)}+(q^{-1}-q)x_\sigma^{(l)}),&  l(\sigma s_i)=l(\sigma)-1, i<l;
			\end{array}
			\right.
		\end{split}
	\end{align*}	
	\begin{align*}
		\Xi({\rm T}_{s_i})\circ\Xi(x_{\sigma}^{(l)})=\Psi(T_{s_i})\Xi_l(x_{\sigma}^{(l)})
		=q^{-1}\Xi_l(x_{\sigma}^{(l)})=\Xi({\rm T}_{s_i}x_{\sigma}^{(l)}) \mbox{ for }\sigma\in \mathfrak{S}_l, i>l;
	\end{align*}	
	\begin{align*}
		\Xi(x_\sigma^{(l)})\circ\Xi({\rm T}_{s_i})=\Xi_l(x_{\sigma}^{(l)}\Psi(T_{s_i})
		=q^{-1}\Xi_l(x_{\sigma}^{(l)})=\Xi(x_{\sigma}^{(l)}{\rm T}_{s_i})\mbox{ for }\sigma\in \mathfrak{S}_l, i>l.
	\end{align*}
	So $\Xi$ is an algebra homomorphism and therefore it is a representation of $\mathfrak{H}\kern -.6em\mathfrak{H}_r$.
\end{proof}
     From the lemma above, we set $\mathcal{D}(n,r):=\Xi(\mathfrak{H}\kern -.6em\mathfrak{H}_r)\subset {\rm End}_\mathscr{A}(\underline{V}^{\otimes r}) $.
\section{Duality related to duplex Hecke algebra}
Firstly, we recall some properties of $q$-Schur algebra.\par
Consider an element $x_{\lambda}=\sum_{\omega\in \mathfrak{S}_{\lambda}}T_{\omega}$ in $\mathbf{H}(\mathfrak{S}_r)$. Then the right ideal $x_{\lambda}\mathbf{H}(\mathfrak{S}_r)$ of $\mathbf{H}(\mathfrak{S}_r)$ is a $\mathbb{Q}(q)$-module, which has a free basis $\{x_{\lambda}T_d\mid d\in \prescript{\lambda}{}{\mathfrak{S}}\}$.
\begin{lem}
	The right $\mathbf{H}(\mathfrak{S}_r)$-module structure on $x_{\lambda}\mathbf{H}(\mathfrak{S}_r)$ is given by the formulas: 
	\begin{align*}
	\begin{split}
		(x_{\lambda}T_d)T_{s_i}=\left\{
		\begin{array}{lll}
		x_{\lambda}T_{ds_i},&l(ds_i)=l(d)+1,ds_i\in\prescript{\lambda}{}{S},\\
		q^{-1}x_{\lambda}T_d,&l(ds_i)=l(d)+1,ds_i\notin\prescript{\lambda}{}{S},\\
				(q^{-1}-q)x_{\lambda}T_d+x_{\lambda}T_{ds_i},&l(ds_i)=l(d)-1,
		\end{array}
			\right.
		\end{split}	\end{align*}
for any $s_i$ $(i=1,2,\cdots,r-1)$ and $d\in \prescript{\lambda}{}{\mathfrak{S}}$. In the third case, when $l(ds_i)=l(d)-1$, we have $ds_i\in\prescript{\lambda}{}{\mathfrak{S}}$.
\end{lem}

\begin{defn}
	For positive integers $n,r$, the Schur algebra over $\mathbb{Q}(q)$ is the endomorphism algebra
     $$S(n,r)={\rm End}_{\mathbf{H}(\mathfrak{S}_r)}(\bigoplus_{\lambda\in\Lambda(n,r)}x_\lambda \mathbf{H}(\mathfrak{S}_r)).$$
    
\end{defn}

\begin{prop}{\label{property}}
	Following the notaion above, we have
	\begin{itemize} 
		\item[(1)] $V^{\otimes r}\simeq \bigoplus_{\lambda\in\Lambda(n,r)}x_\lambda \mathbf{H}(\mathfrak{S}_r)$;	
		\item[(2)] $S(n,r)$ is a free $\mathbb{Q}(q)$-module with a basis
		$\{\zeta^w_{\lambda,\mu}\mid\lambda,\mu\in \Lambda(n,r),w\in\prescript{\lambda}{}{\mathfrak{S}}^{\mu}\}$, where $\zeta^w_{\lambda,\mu}$ only does not annihilate $x_\mu \mathbf{H}(\mathfrak{S}_r)$;
		\item[(3)]If $\mu=(\mu_1,\cdots,\mu_n)$, then $\zeta^w_{\lambda,\mu}$ only has nonzero image on ${V}^{\otimes r}_{k}$ where $k=r-\mu_{n}$.
	\end{itemize} 	
\end{prop}
\begin{proof}
	Result (1)(2) are classical propeties of $q$-Schur algebra, see \cite{DDPW} for  details. Then (3) follows immediately. 
\end{proof}

\begin{defn}
	The subalgebra of $\mathbf{U}_q(\mathfrak{gl}_{n+1})$ generated by $E_i,F_i,H_j,H_j^{-1}(1\leqslant i\leqslant n-1,1\leqslant j\leqslant n+1)$ is called Levi quantum group and denoted by $L_q(\mathfrak{gl}_{n+1})$. It's easy to see $L_q(\mathfrak{gl}_{n+1})\simeq U_q(\mathfrak{gl}_n)\oplus \langle H_{n+1},H_{n+1}^{-1}\rangle$ as a space.
\end{defn}

It's easy to see $L_q(\mathfrak{gl}_{n+1})$ is a Hopf algebra and have a natural representation on $\underline{V}^{\otimes r}$, which we still denote as $\Phi$.

\begin{thm}\label{main}
      Recall that $q$ is transcendental over $\mathbb{Q}$. We have the following double centralizer property:
      $${\rm End}_{\mathcal{D}(n,r)}(\underline{V}^{\otimes r})=\Phi(L_q(\mathfrak{gl}_{n+1})),$$	
      $${\rm End}_{L_q(\mathfrak{gl}_{n+1})}(\underline{V}^{\otimes r})=\mathcal{D}(n,r).$$	
\end{thm}
We prove theorem \ref{main} separately.
\begin{thm}\label{young}
	Keep the notations as above (in particular, $q$ is transcendental over $\mathbb{Q}$). Then we have 
	$${\rm End}_{\mathcal{D}(n,r)}(\underline{V}^{\otimes r})=\Phi(L_q(\mathfrak{gl}_{n+1})).$$
\end{thm}
\begin{proof}
	From lemma \ref{l rep of qha}(3), it is easy to know $\Phi(L_q(\mathfrak{gl}_{n+1}))\subseteq {\rm End}_{\mathcal{D}(n,r)}(\underline{V}^{\otimes r})$. We need to show the opposite inclusion.
	
	Note that $\mathcal{D}(n,r)=\Xi(\mathfrak{H}\kern -.6em\mathfrak{H}_r)$. So ${\rm End}_{\mathcal{D}(n,r)}(\underline{V}^{\otimes r})\subseteq {\rm End}_{\mathbf{H}(\mathfrak{S}_r)}(\underline{V}^{\otimes r})=\Phi(\mathbf{U}_q(\mathfrak{gl}_{n+1}))$.	
	For any $\phi\in \mathrm{End}_{\mathcal{D}(n,r)}(\underline{V}^{\otimes r})$, we have $\phi=\sum_{\lambda,\mu\in\Lambda(n+1,r)}\alpha_{\lambda\mu w}\zeta^w_{\lambda,\mu}$. More precisely, $\phi=\sum_{l=0}^{r}\phi_l$ with $\phi_l=\displaystyle\sum_{\lambda,\mu\in\Lambda(n+1,r), \mu_{n+1}=r-l}\alpha_{\lambda\mu w}\zeta^w_{\lambda,\mu}$. We have $\phi|_{\underline{V}_{\underline{l}}^{\otimes r}}=\phi_l$ by Prop \ref{property}(3).
	
	We claim that $\phi_l$ stabilizes $\underline{V}_{\underline{l}}^{\otimes r}$.
	
    Suppose $\xi_j\in\underline{V}_{\underline{l}}^{\otimes r}, \phi_l(\xi_j)=\xi_l+\xi_k$, with $\xi_l\in\underline{V}_{\underline{l}}^{\otimes r}, \xi_k\notin\underline{V}_{\underline{l}}^{\otimes r}$. Then we have
	$$\psi_{{\rm id}}^{\underline{l}}(\xi_l+\xi_k)=\psi_{{\rm id}}^{\underline{l}}\phi(\xi_j)=\phi \psi_{{\rm id}}^{\underline{l}}(\xi_j)=\xi_l+\xi_k.$$	
	Note that $\psi_{{\rm id}}^{\underline{l}}(\xi_l)=\xi_l$ and $\psi_{{\rm id}}^{\underline{l}}(\xi_k)=0$, which deduce that $\xi_k=0$. Hence $\phi_l$ stabilizes  $\underline{V}_{\underline{l}}^{\otimes r}$.
	
	For any $\underline{V}_I^{\otimes r}$ where $\#I=l$. We have $\omega_I\underline{V}_I^{\otimes r}\simeq\underline{V}_{\underline{l}}^{\otimes r}$. Hence 
	$$\phi(\underline{V}_I^{\otimes r})=\phi(\omega_I^{-1}\underline{V}_{\underline{l}}^{\otimes r})=\omega_I^{-1}\phi(\underline{V}_{\underline{l}}^{\otimes r})\subset \omega_I^{-1}\underline{V}_{\underline{l}}=\underline{V}_I^{\otimes r}.$$
  	Then $\phi$ stabilizes all $\underline{V}_I^{\otimes r}$ such that $\#I=l$, also for $\underline{V}_{\underline{l}}^{\otimes r}$ as desired.
	
	 Notice $\underline{V}_{\underline{l}}^{\otimes r}=V^{\otimes l}\otimes\eta_{n+1}^{\otimes r-l}$ and  $\psi_\sigma^{\underline{l}}=\Psi^V_l(\sigma)\otimes {\rm id}^{\otimes r-l}$.
	    We have  $\phi_l|_{\underline{V}_{\underline{l}}^{\otimes r}}\in\Phi(U_q(\mathfrak{gl}_n))\subset\Phi(L_q(\mathfrak{gl}_{n+1}))$ due to classical quantum $q$-Schur duality. So we can suppose  $\phi_l|_{\underline{V}_{\underline{l}}^{\otimes r}}=\Phi(g_l)|_{\underline{V}_{\underline{l}}^{\otimes r}}$, where $g_l\in L_q(\mathfrak{gl}_{n+1})$.
	  
	We have $\omega_I\phi_l(\underline{V}_I^{\otimes r})=\phi_l\omega_I(\underline{V}_I^{\otimes r})=\Phi(g_l)\omega_I(\underline{V}_I^{\otimes r})$. Since  $\Phi(L_q(\mathfrak{gl}_{n+1}))\subset {\rm End}_{D(n,r)}(\underline{V}^{\otimes r})$, $\Phi(g_l)$ commutes with $\omega_I$. Hence
	\[\phi_l(\underline{V}_I^{\otimes r})=\omega_I^ {-1}\phi_l\omega_I(\underline{V}_I^{\otimes r})=\omega_I^{-1}\Phi(g_l)\omega_I(\underline{V}_I^{\otimes r})=\Phi(g_l)(\underline{V}_I^{\otimes r}).\]
	In other words, $\phi_l|_{\underline{V}_l^{\otimes r}}\in\Phi(L_q(\mathfrak{gl}_{n+1}))|_{\underline{V}_l^{\otimes r}}$.
	
	Next we prove that $\phi\in \Phi(L_q(\mathfrak{gl}_{n+1}))$ for any $\phi\in {\rm End}_{\mathcal{D}(n,r)}(\underline{V}^{\otimes r})$.
	
	Consider the element $G_l=\frac{(H_{n+1}-1)\cdots\widehat{(H_{n+1}-q^{r-l})}\cdots(H_{n+1}-q^r)}{(q^{r-l}-1)\cdots\widehat{(q^{r-l}-q^{r-l})}\cdots(q^{r-l}-q^r)}$, where  $\widehat{a}$ means ignoring this term. Then the action of $G_l$ is:
	$$G_l.x=\frac{(q^{r-l}-1)\cdots\widehat{(q^{r-l}-q^{r-l})}\cdots(q^{r-l}-q^r)}{(q^{r-l}-1)\cdots\widehat{(q^{r-l}-q^{r-l})}\cdots(q^{r-l}-q^r)}x=x, \forall x\in\underline{V}_l^{\otimes r};$$
	$$G_l.\underline{V}_k^{\otimes r}=0\ (k\neq l),\text{ since }(H_{n+1}-q^{r-k}).\underline{V}_k^{\otimes r}=0.$$
	Then we have
	$$\phi(g_l)\Phi(G_l). \underline{V}_l^{\otimes r}=\phi(g_l). \underline{V}_l^{\otimes r};$$
	$$\phi(g_l)\Phi(G_l).\underline{V}_k^{\otimes r}=0\ (k\neq l).$$
	Thus 
	$\phi_l=\Phi(g_l\circ G_l)$ and $\phi=\sum\limits_l\Phi(g_l\circ G_l)\in\Phi(L_q(\mathfrak{gl}_{n+1}))$.
\end{proof}

We now prepare for the remaining part of Theorem \ref{main}. Notice that we have $\underline{V}^{\otimes r}_{\underline{l}}\simeq V^{\otimes l}$, which can be regarded as a $(\mathbf{U}_q(\mathfrak{gl}_n),\mathbf{H}(\mathfrak{S} _l))$-module. 
Due to quantum $q$-Schur duality, we have $\underline{V}^{\otimes r}_{\underline{l}}\simeq V^{\otimes l}\simeq \bigoplus_{\lambda\in P{(n,l)}}L_q(\lambda)\otimes S_q^{\lambda}$ follows from Theorem \ref{q duality}(2). Hence we have 
$$\underline{V}^{\otimes r}_{\underline{l}}\simeq V^{\otimes l}\simeq\bigoplus_{\lambda\in P{(n,l)}}\overbrace{(S_q^{\lambda}\oplus S_q^{\lambda}\cdots\oplus S_q^{\lambda})}^{\mathrm{dim}(L_q(\lambda))}.$$
We thereby obtain an irreducible decomposition of the $\mathbf{H}(\mathfrak{S} _l)$-modules $\underline{V}^{\otimes r}_{\underline{l}}$.
The image of $S_q^{\lambda}$ in $\underline{V}^{\otimes r}_{\underline{l}}$ is
$${S'}_q^{\lambda}:=S_q^ {\lambda}\otimes\overbrace{\eta_{n+1}\otimes\cdots\otimes\eta_{n+1}}^{r-l}.$$
We consider the space $D^\lambda_l:=\sum_{I}\omega_I^{-1}(S')_q^{\lambda}$ with $\#I=l$  (recall that $\omega_I\in\mathfrak{S}_r$ satisfies  $\Psi(\omega_I)\underline{V}^{\otimes r}_{I}\subset\underline{V}^{\otimes r}_{\underline{l}}$).

\begin{lem}\label{keylemma}
	$D^\lambda_l$ is an irreducible $\mathfrak{H}\kern -.6em\mathfrak{H}_r^l$-module. In particular, it is an  irreducible $\mathfrak{H}\kern -.6em\mathfrak{H}_r$-module.
\end{lem}
\begin{proof}	
	Firstly, we show $D^\lambda_l$ is a $\mathfrak{H}\kern -.6em\mathfrak{H}_r^l$-module. Note that the generator of $\mathfrak{H}\kern -.6em\mathfrak{H}_r^l$ are $T_{s_i}$ and $x_\sigma$, where $1\leq s\leq r$ and $x_\sigma\in \mathfrak{S}_l$. From the natural representation, $x_\sigma$ only doesn't annihilate $\underline{V}^{\otimes r}_{\underline{l}}$ and obviously stabilizes $D^\lambda_l$. Note that $\omega_I$ is generated by $T_{s_i}$ and thus $D^\lambda_l$ stabilized by all $T_{s_i}$. We conclude that $D^\lambda_l$ is a $\mathfrak{H}\kern -.6em\mathfrak{H}_r^l$-module. The irreducibility is immediate from the definition. 
   
\end{proof}

\begin{prop}
	Suppose $q$ is transcendental over $\mathbb{Q}$. We have 
	$$\underline{V}^{\otimes r}_l\simeq\bigoplus_{\lambda\in P(n,l)} D_l^\lambda\otimes L_q(\lambda)$$
	as a $(\mathfrak{H}\kern -.6em\mathfrak{H}_r,L_q(\mathfrak{gl}_{n+1}))$-module.
\end{prop}
\begin{proof}
	We can easily get $\underline{V}^{\otimes r}_l\simeq \sum_I \omega_I^{-1}\underline{V}^{\otimes r}_I$ from the action of $\mathfrak{H}\kern -.6em\mathfrak{H}_r^l$.
	 Note that $L_q(\lambda)$ is an irrecudible $L_q(\mathfrak{gl}_{n+1})$-module. Thus we have desried decomposition of  $\underline{V}^{\otimes r}_l$ as a $(\mathfrak{H}\kern -.6em\mathfrak{H}_r,L_q(\mathfrak{gl}_{n+1}))$-module.
\end{proof}

\begin{coro}
     Suppose $q$ is transcendental over $\mathbb{Q}$. Then we have 
	$${\rm End}_{L_q(\mathfrak{gl}_{n+1})}(\underline{V}^{\otimes r})=\mathcal{D}(n,r).$$
\end{coro}
\begin{proof}
	From \ref{keylemma}, $D^\lambda_l$ is a irreducible $\mathfrak{H}\kern -.6em\mathfrak{H}_r^l$-module and also a irreducible $\mathfrak{H}\kern -.6em\mathfrak{H}_r$-module. So
	$$\underline{V}^{\otimes r}\simeq\bigoplus\limits_{l=0}^r\bigoplus_{\lambda\in P(n,l)}(D_l^\lambda)^{\bigoplus \mathrm{dim}(L_q(\lambda))}.$$
	is the decomposition of $\underline{V}^{\otimes r}$ into direct sum of irreducible $\mathfrak{H}\kern -.6em\mathfrak{H}_r$-modules. From classical duality theory (see \cite{GW} for details), we have  $${\rm End}_{L_q(\mathfrak{gl}_{n+1})}(\underline{V}^{\otimes r})=\mathcal{D}(n,r).$$
\end{proof}

\subsection*{Acknowledgements}
We are grateful to professor Tanisaki for his enlightenment of quantum group and professor Bin Shu for his suggestions and helps about this article.

\end{document}